\documentclass[11pt,a4paper]{article}
\usepackage{geometry}
\usepackage{hyperref}
\usepackage{amsmath}
\usepackage{amssymb}
\usepackage{amsthm}
\usepackage{graphicx}
\usepackage{subfig}
\usepackage{svg}
\usepackage{tikz}
\usepackage{tikz-cd}
\usepackage{shuffle}

\theoremstyle{plain}
\newtheorem{Th}{Theorem}[section]
\newtheorem{Lemma}[Th]{Lemma}

\newtheorem{Prop}[Th]{Proposition}

\theoremstyle{definition}
\newtheorem{Def}[Th]{Definition}

\newtheorem{Rem}[Th]{Remark}
\newtheorem{?}[Th]{Problem}
\newtheorem{Ex}[Th]{Example}

\DeclareSymbolFont{rsfs}{U}{rsfs}{m}{n}
\DeclareSymbolFontAlphabet{\mathscrsfs}{rsfs}

\newcommand{\ldb}{\{\!\!\{}
\newcommand{\rdb}{\}\!\!\}}

\begin{document}
	\title{A relation between Turaev coaction, Goncharov--Brown coaction and the reduced coaction Lie algebra}
	\author{Muze Ren\thanks{Section of Mathematics, University of Geneva, Rue du Conseil-Général 7-9, 1205 Geneva, Switzerland \href{mailto:muze.ren@unige.ch}{muze.ren@unige.ch}}} 
	\maketitle	
	\begin{abstract}
	 We present a formula that relates the Turaev coaction and the Goncharov–Brown coaction. Motivated by this relation, we introduce the reduced coaction equation. The skew-symmetric solutions to this equation form a Lie algebra under Ihara bracket.	
    \end{abstract}

\section{Introduction and main results}
Let $k$ be a field of characteristic $0$. We are interested in the following two different Hopf algebras of formal noncommutative power series:
\[A=(k\langle\langle x_0,x_1\rangle\rangle,\rm{conc},\Delta_{\shuffle}, \varepsilon, S)\quad \rm{and}\quad \widehat{A}=(k\langle\langle x_0,x_1\rangle\rangle, \shuffle,\Delta_{\rm{dec}}, \varepsilon).\]
Here, "$\rm{conc}$" denotes the concatenation product, and $\shuffle$ is the shuffle product. The counit is defined by  $\varepsilon(x_0)=\varepsilon(x_1)=0$. The coproducts are given by
\begin{align*}
&\Delta_{\shuffle}(x_i)=x_i\otimes 1+1\otimes x_i,\quad i=0,1.
&\Delta_{\rm{dec}}(w)=\sum_{uv=w}u\otimes v.
\end{align*}

The Hopf algebra $A$ is dual to $(k\langle x_0,x_1\rangle, \shuffle, \Delta_{\rm{dec}})$ with respect to the pairing \begin{equation*}
k\langle\langle x_0,x_1\rangle\rangle\otimes k\langle x_0,x_1\rangle; \varphi\otimes w\mapsto c_{w}(\varphi),
\end{equation*}
where $c_{w}(\varphi)$ denotes the coefficient of the word $w$ in the series $\varphi$. For any $x\in A$, we can decompose it as
\begin{equation*}
	x=\varepsilon(x)+\sum^1_{i=0}x_id^R_i(x)=\varepsilon(x)+\sum^1_{i=0}d^{L}_i(x)x_i
\end{equation*}

We denote the free Lie algebra in $A$ by $\mathfrak{fr}_k(x_0,x_1)$, the indecomposable $\ker \varepsilon/(\ker\varepsilon)^2$ of $\widehat{A}$ by $\mathcal{L}$.  

There are two algebraic operations on $A$ and $\widehat{A}$ that are dual to each other, namely the Ihara action and the Goncharov-Brown coaction. These operations play a fundamental role in the study of multiple zeta values, mixed Tate motives and polylogarithms.

\begin{itemize}
	\item {\bf{Ihara action}} $I:\mathfrak{fr}_k(x_0,x_1)\times A\to A; \psi\times a \mapsto d_{\psi}(a)$, 	
	where $d_{\psi}$ is the derivation given by $d_{\psi}(x_0)=0$ and $d_{\psi}(x_1)=[x_1,\psi]$. This action induces the Ihara Lie bracket on $\mathfrak{fr}_k(x_0,x_1)$ given by
	\begin{equation}\label{eq:Ihara_bracket}
		\{\psi_1,\psi_2\}=d_{\psi_2}(\psi_1)-d_{\psi_1}(\psi_2)-[\psi_1,\psi_2]
	\end{equation}
    For the arithmetic motivations, see \cite{Ihara} and \cite{Deligne1989}. 
    \item {\bf{Goncharov-Brown} coaction} $D:\widehat{A}\to \mathcal{L}\otimes \widehat{A}$. 
    \begin{align}\label{eq:GBcoaction}
    	&D(\varepsilon_1\ldots \varepsilon_n)=\sum_{0\le p<q\le n}[I(\varepsilon_p;\varepsilon_{p+1}\dots\varepsilon_{q};\varepsilon_{q+1})]\otimes \varepsilon_1\dots \varepsilon_p\varepsilon_{q+1}\dots \varepsilon_n
    \end{align}
    (The notation will be clarified in next section.) This coaction induces a Lie cobracket on $\mathcal{L}$. For the definition of coproduct, see \cite{Goncharov2001}. For the important application of the coaction, \cite{Brown2012,Brown2013,Brown2014}. 
\end{itemize}

In the study of intersections of curves on surfaces, the Goldman--Turaev Lie bialgebra \cite{Turaev} and its noncommutative analogue, the necklace Lie bialgebras \cite{Schedler} play an important role. Those two brackets also arise from another pair of action and coaction. The  Kirillov-Kostant-Souriau double bracket $\ldb-,-\rdb_{\rm{KKS}}$ induces the following action.
\begin{itemize}
	\item {\bf{KKS action}}: $p:A\times A\to A$ which descends to $A/[A,A]\times A\to A$ and induces the necklace Lie bracket
	$\{-,-\}_{\rm{necklace}}$ on $A/[A,A]$, it induces the KKS poisson bracket on representation space.  See section 2.6 in \cite{VandenBergh2008} and section 3 in \cite{AFR2024} for further details.\\
	\item {\bf{Turaev coaction}}: The reduced coaction $\mu:A\to A$. It is equivalent to the coaction map $p_{\mu}:A\to  A/[A,A]\otimes A$ and the coaction induces a Lie cobracket $\delta_{\rm{necklace}}$ on $A/[A,A]$. 
	
	We discuss only the reduced coaction in this paper. For further details of the coaction, see the linear term of the Turaev coproduct in \cite{Turaev}, "$\mu$" in \cite{KK2014} and "quasi derivations" in \cite{Massuyeau}. For the equivalence to reduced one, see section 3 of \cite{AFR2024}.
\end{itemize}

The relation between Ihara action and KKS action was first studied by Drinfeld in the context of Hamiltonian interpretation of Ihara Lie algebra (section 6 in \cite{Drinfeld1990}). And Kontsevich \cite{Kontsevich} using noncommutative geoemtry, see also Goncharov Proposition 5.1 in \cite{Goncharov2005} etc. It is then very natural to compare the coactions.

The reduced coaction $\mu$ is a linear map from $k\langle\langle x_0,x_1\rangle\rangle$ to $k\langle\langle x_0,x_1\rangle\rangle$ which is defined to be
\begin{align*}
		&\mu(x_0)=\mu(x_1)=0,\quad
		(x_i\odot x_{j}):=\delta_{x_i,x_j}x_i,\quad x_i,x_j\in \{x_0,x_1\}\\
		&\mu(k_1k_2,\dots,k_n):=\sum^{n-1}_{i=1}k_1,\dots,k_{i-1}(k_i\odot k_{i+1})k_{i+2}\dots k_n,\quad k_1,\dots,k_n\in \{x_0,x_1\}
\end{align*}
We denote $\widehat{\mu}$ to be the dual of $\mu$. Our first result relating the coaction $D$ and $\widehat{\mu}$ is the following simple explicit formula, see next section for notation.

\begin{Th}
For $\varepsilon_1,\dots,\varepsilon_n\in \{x_0,x_1\}$, we have
\begin{multline}\label{eq:commute}
D\circ \widehat{\mu}-(1\otimes \widehat{\mu})\circ D (\varepsilon_1\ldots \varepsilon_n)\\
=\sum_{0\le p<q\le n}[I^{\mu}(\varepsilon_p;\varepsilon_{p+1}\dots\varepsilon_{q};\varepsilon_{q+1})]\otimes \varepsilon_1\dots \varepsilon_p\varepsilon_{q+1}\dots \varepsilon_n
\end{multline}
\end{Th}

The formula is interesting because we could denote the RHS of \eqref{eq:commute} by $D^{\mu}$. The definition of $D^{\mu}$ looks similar to the definition of $D$ with the help of the notations $I^{\mu}$. This formula also allow us to further compare the two co-products.  And inspired by the notation $I^{\mu}$, we introduce the following form of reduced coaction equations and a Lie algebra defined with Ihara bracket.

Let $\eta$ be a Lie series in $k\langle\langle x_0,x_1\rangle\rangle$, with $c_{x_0}(\eta)=c_{x_1}(\eta)=0$, the function $r_{\eta}$ associated to $\eta$ is defined to be $r_{\eta}(x)=\sum_{l\ge 0}c_{x^{l+1}_0x_1}(\eta)x^{l+1}$.

\begin{Def}[The reduced coaction equation] For $\eta$ in $\mathfrak{fr}_k(x_0,x_1)$ with $c_{x_0}(\eta)=c_{x_1}(\eta)=0$. The reduced coaction equation is defined to be
\begin{equation}\label{variant_reduced_coaction_equation}
d^R_1(\eta)+\mu(\eta)+d^L_0(\eta)=-r_{\eta}(x_1)+r_{\eta}(x_0).
\end{equation}

We denote by $\overline{{\mathfrak{rc}}_0}$ the vector space of skew symmetric solutions of the equation, i.e. $\eta(x_0,x_1)=-\eta(x_1,x_0)$.
\end{Def}

\begin{Ex}
	For the element	$f_3=[x_0,[x_0,x_1]]+[x_1,[x_0,x_1]]$, we have that $\mu(f_3)=0$ and 
	\begin{equation*}
		d^L_0(f_3)=-2x_0x_1+x_1x_0-x_1^2,\quad  d^R_1(f_3)=x_0^2+2x_0x_1-x_1x_0, \quad r_{f_3}(x)=x^2.
	\end{equation*}
	And we can directly verify that $f_3$ satisfies the reduced coaction equation 
	\begin{equation}
	d^L_0(f_3)+\mu(f_3) +d^R_1(f_3)	=-r_{f_3}(x_1)+r_{f_3}(x_0)
	\end{equation}
\end{Ex}

\begin{Th}\label{th:Ihara_bracket}
	The vector space $\overline{{\mathfrak{rc}}_0}$ is a Lie algebra with the Ihara bracket \eqref{eq:Ihara_bracket},
\end{Th}

\vspace{0.4 cm}
\begin{Rem}
For $\Phi$ a group-like element in $A$, the group version of the reduced coaction equation is the following
\begin{equation}\label{eq:group_reduced_coaction_ equation}
	d^{R}_1(\Phi)+\mu(\Phi)+d^{L}_0(\Phi)=-r_{\phi}(x_1)\Phi+\Phi r_{\phi}(-x_0)
\end{equation}
the infinitesimal equation is actually
\begin{equation}\label{eq:reduced_coaction_equation}
	d^R_1(\eta)+\mu(\eta)+d^L_0(\eta)=-r_{\eta}(x_1)+r_{\eta}(-x_0),
\end{equation}

The skew symmetric solutions of the \eqref{eq:reduced_coaction_equation} is denoted by $\mathfrak{rc}$. $\mathfrak{rc}$ contains a 1 dimensional vector space $[x_0,x_1]$, the subspace $\mathfrak{rc}_\lambda$ whose cofficents before the commutator $[x_0,x_1]$ is $\lambda$. It is  proved in \cite{AHNRS2025}, for $\eta\in \mathfrak{rc}_0$, the functions $r_{\eta}(x)$ is an even function. It is then equivalent to \eqref{variant_reduced_coaction_equation}.

The Lie algebra $\mathfrak{rc}_0$ has a close relation with the double shuffle Lie algebra $\mathfrak{dmr}_0$, Kashiwara-Vergne Lie algebra $\mathfrak{krv}_2$ and Grothendieck Teichmuller Lie algebra $\mathfrak{grt}_1$. As a direct consequence of Theorem 1.1 in \cite{AFR2024}, $\mathfrak{grt}_1$ satisfies \eqref{eq:reduced_coaction_equation} and $r_{\eta}$ is even function, thus $\mathfrak{grt}_1\in \overline{{\mathfrak{rc}}_0}.$ The detailed relations are studied in \cite{AHNRS2025}, and conjecturally we have the isomorphism between the Lie algebras $\mathfrak{grt}_1$, $\mathfrak{dmr}_0$, $\mathfrak{krv}_2$ and $\mathfrak{rc}_0$. Similar objects in \cite{AFR2024} are studied from the topological perpective in \cite{DrorYusuke}\footnote{It will be interesting to compare with $\mathfrak{grt}^{\rm{em}}_1$ in \cite{DrorYusuke} or symmetric $\mathfrak{krv}_2$ which use krv1 equation.}  and categorical perspective in \cite{RodrigoNaef}. We hope to investigate the relation to those perspective in the future. 
\end{Rem}

{\bf Acknowledgements.} I would like to thank A.~Alekseev, F.~Brown and F.~Naef, for teaching me those coactions. And thank D.~Bar-Natan, H.~Furusho, Y.~Kuno, G.~Massuyeau and P.~Severa  for inspiring questions and explain some questions in their papers. Part of the work is based on the talk "From generalized pentagon equations to reduced coaction equations" given at Dijon. I would like to thank M.~Fairon and T.~Kimura for invitation and hospitality.

This work is supported in part by the grants 08235 and 20040, and by the National Center for Competence in Research SwissMAP of the Swiss National Science Foundation.

\section{Commuting relation of two coactions}
We first recall Brown's formula for the coaction and its proof in Proposition 2.6 in \cite{Brown2013}. Then we use the same idea to have the formula for $D\circ \widehat{\mu}-(1\otimes \widehat{\mu})\circ D$.

For the consideration of the dual operation, the Ihara action $d_{\psi}$ can be written as sum of the following cases, namely multiply $S(\psi)$ to the words that begin with $x_1$, insert $\psi$ to every adjacent $x_0x_1$, $S(\psi)$ to every adjacent $x_1x_0$ and multiply $\psi$ to the words that ends in $x_1$. More formally,
\begin{align*}
&(1)  \{\varepsilon_1\ldots\varepsilon_n|I(\psi;x_1)\}:={\bf{S(\psi)x_1}}d^{R}_1(\varepsilon_1\ldots \varepsilon_n)\\
&(2) \{\varepsilon_1\ldots\varepsilon_n|I(x_0;\psi;x_1)\}:=\sum_{\substack{1\le i\le n-1,\varepsilon_i=x_0,\varepsilon_{i+1}=x_1}}\varepsilon_1\dots\varepsilon_{i-1}{\bf{x_0S(\psi)x_1}}\varepsilon_{i+2}\dots\varepsilon_n\\
&(3) \{\varepsilon_1\ldots\varepsilon_n|I(x_1;\psi;x_0)\}:=\sum_{\substack{1\le i\le n-1\varepsilon_i=x_1,\varepsilon_{i+1}=x_0}}\varepsilon_1\dots\varepsilon_{i-1}{\bf{x_1\psi x_0}}\varepsilon_{i+2}\dots\varepsilon_n\\
&(4) \{\varepsilon_1\ldots\varepsilon_n|I(x_0;\psi;x_0)\}:=\sum_{\substack{1\le i\le n-1\varepsilon_i=x_0,\varepsilon_{i+1}=x_0}}\varepsilon_1\dots\varepsilon_{i-1}{\bf{x_0 0 x_0}}\varepsilon_{i+2}\dots\varepsilon_n\\
&(5) \{\varepsilon_1\ldots\varepsilon_n|I(x_1;\psi;x_1)\}:=\sum_{\substack{1\le i\le n-1\varepsilon_i=x_1,\varepsilon_{i+1}=x_1}}\varepsilon_1\dots\varepsilon_{i-1}{\bf{x_1 0 x_1}}\varepsilon_{i+2}\dots\varepsilon_n\\
&(6) 
\{\varepsilon_1\ldots\varepsilon_n|I(x_1;\psi)\}:=d^{L}_1(\varepsilon_1\ldots \varepsilon_n){\bf{x_1\psi}}\\
\end{align*}

The Goncharov-Brown coaction $D$ is the operation dual to the Ihara action map. We first recall the following notation  
\begin{equation}
	I(\varepsilon_p;f;\varepsilon_q)=\begin{cases}
		&S(f),\quad (\varepsilon_p,\varepsilon_q)=(x_0,x_1)\\
		&f,\quad (\varepsilon_p,\varepsilon_q)=(x_1,x_0)\\
		&\varepsilon(f) \quad \varepsilon_p=\varepsilon_q
	\end{cases}
\end{equation}

\begin{Prop}[\cite{Brown2014},Proposition 2.6]
	For $\varepsilon_1,\dots,\varepsilon_n\in \{x_0,x_1\}$, the explicit formula for the coaction $D:\widehat{A} \to \mathcal{L}\otimes \widehat{A}$ is
\begin{align}
	&D(\varepsilon_1\ldots \varepsilon_n)=\sum_{0\le p<q\le n}[I(\varepsilon_p;\varepsilon_{p+1}\dots\varepsilon_{q};\varepsilon_{q+1})]\otimes \varepsilon_1\dots \varepsilon_p\varepsilon_{q+1}\dots \varepsilon_n
\end{align}
here the square bracket $[~]$ is the projection from $\widehat{A} \to \mathcal{L}$, we also set $\varepsilon_0=x_1,\varepsilon_{n+1}=x_0$. 
\end{Prop}
 
\begin{proof}
The defintion of the notation $I(\varepsilon_p;\varepsilon_{p+1}\dots\varepsilon_{q};\varepsilon_{q+1})$ is dual to the above cases (2),(3),(4),(5). And the requirement $\varepsilon_0=x_1,\varepsilon_{n+1}=x_0$ is dual to the cases (1) and (6).
\end{proof}

We follow the same idea to study the operation $d^{\mu}=\mu\circ \rm{I}-I\circ (1\otimes \mu)$, this map from $\mathfrak{fr}_k( x_0,x_1)\times A$ to $A$ can be written as sum of the following cases, for an element $\psi\in \mathfrak{fr}_k(x_0,x_1)$, $d^{\mu}_{\psi}$ is written as
\begin{align*}
	&(1)  \{\varepsilon_1\ldots\varepsilon_n|I^{\mu}(\psi;x_1)\}:={\bf{(\mu(S(\psi))+d^{L}_1(S(\psi)))}x_1}d^{R}_1(\varepsilon_1\ldots \varepsilon_n)\\
	&(2) \{\varepsilon_1\ldots\varepsilon_n|I^{\mu}(x_0;f;x_1)\}:=\\
	&\sum_{\substack{1\le i\le n-1,\varepsilon_i=x_0,\varepsilon_{i+1}=x_1}}\varepsilon_1\dots\varepsilon_{i-1}{\bf{x_0(d^R_0(S(\psi))+\mu(S(\psi))+d^L_1(S(\psi)))x_1}}\varepsilon_{i+2}\dots\varepsilon_n\\
	&(3) \{\varepsilon_1\ldots\varepsilon_n|I^{\mu}(x_1;f;x_0)\}:=\\
	&\sum_{\substack{1\le i\le n-1\varepsilon_i=x_1,\varepsilon_{i+1}=x_0}}\varepsilon_1\dots\varepsilon_{i-1}{\bf{x_1(d_1^{R}(\psi)+\mu(\psi)+d^L_0(\psi)) x_0}}\varepsilon_{i+2}\dots\varepsilon_n\\
	&(4) \{\varepsilon_1\ldots\varepsilon_n|I^{\mu}(x_0;\psi;x_0)\}:=\sum_{\substack{1\le i\le n-1\varepsilon_i=x_1,\varepsilon_{i+1}=x_0}}\varepsilon_1\dots\varepsilon_{i-1}{\bf{x_0 0 x_0}}\varepsilon_{i+2}\dots\varepsilon_n\\
	&(5) \{\varepsilon_1\ldots\varepsilon_n|I^{\mu}(x_1;\psi;x_1)\}:=\sum_{\substack{1\le i\le n-1\varepsilon_i=x_1,\varepsilon_{i+1}=x_0}}\varepsilon_1\dots\varepsilon_{i-1}{\bf{x_1 0 x_1}}\varepsilon_{i+2}\dots\varepsilon_n\\
	&(6) 
	\{\varepsilon_1\ldots\varepsilon_n|I^{\mu}(x_1;\psi)\}:=d^{L}_1(\varepsilon_1\ldots \varepsilon_n){\bf{x_1(d^R_1(\psi)+\mu(\psi))}}\\
\end{align*}

We introduce the following notation:
\begin{equation}\label{eq:definition_mu_I}
	I^{\mu}(\varepsilon_p;f;\varepsilon_q)=\begin{cases}
		&(\mu(S(f))+d^{L}_1(S(f))),\quad p=0\\
		&d^R_0(S(f))+\mu(S(f))+d^L_1(S(f)),\quad (\varepsilon_p,\varepsilon_q)=(x_0,x_1)\\
		&d_1^{R}(f)+\mu(f)+d^L_0(f),\quad (\varepsilon_p,\varepsilon_q)=(x_1,x_0)\\
		&\varepsilon(f) \quad \varepsilon_p=\varepsilon_q\\
		&(d^R_1(f)+\mu(f)),\quad q=n+1
	\end{cases}
\end{equation} 

The map dual to $d^{\mu}:\mu\circ \rm{I}-\rm{I}\circ (1\otimes \mu)$ is the map $D^{\mu}:D\circ \widehat{\mu}-(1\otimes \widehat{\mu})\circ D$ and it has the following formula.
\begin{Th}
For $\varepsilon_1,\dots,\varepsilon_n\in \{x_0,x_1\}$, we have
\begin{multline}
D\circ \widehat{\mu}-(1\otimes \widehat{\mu})\circ D (\varepsilon_1\ldots \varepsilon_n)\\
=\sum_{0\le p<q\le n}[I^{\mu}(\varepsilon_p;\varepsilon_{p+1}\dots\varepsilon_{q};\varepsilon_{q+1})]\otimes \varepsilon_1\dots \varepsilon_p\varepsilon_{q+1}\dots \varepsilon_n
\end{multline}
\end{Th}
\begin{proof}
The defintion of the notation $I^{\mu}(\varepsilon_p;\varepsilon_{p+1}\dots\varepsilon_{q};\varepsilon_{q+1})$ is dual to the 6 cases of the action $\mu\circ \rm{I}-\rm{I}\circ (1\otimes \mu)$.
\end{proof}

\section{Reduced coaction Lie algebra}
We prove the Theorem \ref{th:Ihara_bracket} in this section. 

\begin{Lemma}
	By the skew-symmetric property, for $\eta\in \overline{{\mathfrak{rc}}_0}$, it satisfies
	\begin{equation}\label{eq:skew_reduced_equation}
		d^{R}_0(\eta)+\mu(\eta)+d^{L}_1(\eta)=-r_{\eta}(x_1)+r_{\eta}(x_0)
	\end{equation}
\end{Lemma}
\begin{proof}
The operation $\mu$ commutes with the algebra automorphism $\tau$ that exchange $x_0$ and $x_1$, $\tau\circ \mu=\mu\circ \tau$. We apply $\tau$ to the equation \eqref{eq:reduced_coaction_equation}. With the skew-symmetric property, we have $\tau\circ d^{R}_1(\eta)=-d^R_0(\eta)$, $\tau\circ d^L_0(\eta)=-d^L_1(\eta)$.
\end{proof}

With the notations induced as before, we have
\begin{equation*}
d_{\psi_2}(\psi_1)=\{\psi_1|I(\psi_2;x_1)\}+\{\psi_1|I(x_0;\psi_2;x_1)\}+\{\psi_1|I(x_1;\psi_2;x_0)\}+\{\psi_1|I(x_1;\psi_2)\}
\end{equation*}
\begin{equation*}
d^{\mu}_{\psi_2}(\psi_1)=\{\psi_1|I^{\mu}(\psi_2;x_1)\}+\{\psi_1|I^{\mu}(x_0;\psi_2;x_1)\}+\{\psi_1|I^{\mu}(x_1;\psi_2;x_0)\}+\{\psi_1|I^{\mu}(x_1;\psi_2)\}
\end{equation*}

\begin{Lemma}\label{lemma:calculation_d_mu}
Let $\psi\in \overline{{\mathfrak{rc}}_0}$ and $f\in A$ with $c_{x^n_0}(f)=c_{x^n_1}(f)=\varepsilon(f)=0$, $n\ge 1$, we have
\begin{align*}
&\{f|I^{\mu}(x_0;\psi;x_1)\}+\{f|I^{\mu}(x_1;\psi;x_0)\}\\
&=-r_{\psi}(x_1)x_1d^R_1(f)+d^L_1(f)x_1 r_{\psi}(x_1)-r_{\psi}(x_0)x_0d^R_{0}(f)+d^L_0(f)x_0r_{\psi}(x_0)\\
&\{f|I^{\mu}(x_1;\psi)\}=d^L_1(f)x_1(-r_{\psi}(x_1)+r_{\psi}(x_0)-d^L_0(\psi)-d^R_1(\psi))+d^L_1(f)x_1d^R_1(\psi)\\
&\{f|I^{\mu}(\psi;x_1)\}=-(-r_{\psi}(x_1)+r_{\psi}(x_0)-d^L_0(\psi)-d^R_1(\psi))x_1d^R_1(f)-d^L_1(\psi)x_1d^R_1(f)\\
\end{align*}
\end{Lemma}

\begin{proof}
\begin{align*}
&\{f|I^{\mu}(x_0;\psi;x_1)\}+\{f|I^{\mu}(x_1;\psi;x_0)\}\\
&=\{f|I(x_0; -d^R_0(\psi)-\mu(\psi)-d^L_1(\psi);x_1)\}+\{f|I(x_1;d^R_1(\psi)+\mu(\psi)+d^L_0(\psi));x_0\}\\
&=\{f|I(x_0; r_{\psi}(x_1)-r_{\psi}(x_0);x_1)\}+\{f|I(x_1;-r_{\psi}(x_1)+r_{\psi}(x_0);x_0\}\\
&=-r_{\psi}(x_1)x_1d^R_1(f)+d^L_1(f)x_1 r_{\psi}(x_1)-r_{\psi}(x_0)x_0d^R_{0}(f)+d^L_0(f)x_0r_{\psi}(x_0).
\end{align*}
The second equility is by definition, the third equility is by \eqref{eq:reduced_coaction_equation} and \eqref{eq:skew_reduced_equation}, the last equility is because $c_{x^n_0}(f)=0,c_{x^n_1}(f)=0$, then the functions $r_{\psi}(x)$ commutes to the leftmost and rightmost.

For the other two equations, $\{f|I^{\mu}(x_1;\psi)\}=d^L_1(f)x_1(\mu(\psi)+d^R_1(\psi))$ and $\{f|I^{\mu}(\psi;x_1)\}=(-d^L_1(\psi)-\mu(\psi))x_1d^R_1(f)$
\end{proof}

\begin{Lemma}\label{lemma:calculation_mu}
Let $\psi\in \overline{{\mathfrak{rc}}_0}$ and $f\in \mathfrak{fr}_k(x_0,x_1)$, we have
\begin{align*}
d_{f}(\mu(\psi))&=d_{f}(-r_{\psi}(x_1)+r_{\psi}(x_0)-d^L_0(\psi)-d^R_1(\psi))\\
&=[-r_{\psi}(x_1),f]-d_{f}(d^L_0(\psi)+d^R_1(\psi)),
\end{align*}
\end{Lemma}

The notations $d^R_i$ and $d^L_i$ are similar to that of $d_{\psi}$, to avoid confusion in the following calculations, we also use the following notation. For any $x\in A$, we write
\begin{equation*}
	x=\varepsilon(x)+(x)_{x_0}x_0+(x)_{x_1}x_1=\varepsilon(x)+x_0{}_{x_0}(x)+x_1{}_{x_1}(x),
\end{equation*}
namely $d^R_i(x)={}_{x_i}(x)$ and $d^L_i(x)=(x)_{x_i}$.

\begin{Lemma}\label{lem:mid_1}
	Suppose that $\psi_1,\psi_2$ are in $\overline{{\mathfrak{rc}}_0}$, then we have
	\begin{align*}
		\mu\circ d_{\psi_2}(\psi_1)&=[\psi_1,r_{\psi_2}(x_0)]+[-r_{\psi_1}(x_1),\psi_2]-d_{\psi_2}((\psi_1)_{x_0}+{}_{x_1}(\psi_1))\\
		&+(\psi_2)_{x_0}x_1{}_{x_1}(\psi_1)+{}_{x_1}(\psi_2)x_1{}_{x_1}(\psi_1)-(\psi_2)_{x_1}x_1{}_{x_1}(\psi_1)\\
		&-(\psi_1)_{x_1}x_1(\psi_2)_{x_0}
	\end{align*}
\end{Lemma}

\begin{proof}
We have that $\mu\circ d_{\psi_2}(\psi_1)=d_{\psi_2}(\mu(\psi_1))+d^{\mu}_{\psi_2}(\psi_1)$. The terms $d_{\psi_2}(\mu(\psi_1))$ is calculated in Lemma \ref{lemma:calculation_mu} and $d^{\mu}_{\psi_2}(\psi_1)$ is calculated in the Lemma \ref{lemma:calculation_d_mu}.
\end{proof}

\begin{Lemma}\label{lem:mid_2}
	For $\psi_1,\psi_2\in \overline{{\mathfrak{rc}}_0}$, $r_{\{\psi_1,\psi_2\}}=0$ and
	\begin{align*}
		\{\psi_1,\psi_2\}_{x_0}=&d_{\psi_2}((\psi_1)_{x_0})+(\psi_1)_{x_1}x_1(\psi_2)_{x_0}-d_{\psi_1}((\psi_2)_{x_0})-(\psi_2)_{x_1}x_1(\psi_1)_{x_0}\\
		&-\psi_1(\psi_2)_{x_0}+\psi_2(\psi_1)_{x_0}\\
		{}_{x_1}\{\psi_1,\psi_2\}=&\psi_2~{}_{x_1}(\psi_1)-{}_{x_1}(\psi_2)x_1{}_{x_1}(\psi_1)+d_{\psi_2}({}_{x_1}(\psi_1))\\
		&-\psi_1~{}_{x_1}(\psi_2)+{}_{x_1}(\psi_1)x_1{}_{x_1}(\psi_2)-d_{\psi_1}({}_{x_1}(\psi_2))\\
		&-{}_{x_1}(\psi_1)\psi_2+x_1(\psi_2)\psi_1
	\end{align*}
\end{Lemma}

\begin{proof}
	The series $\{\psi_1,\psi_2\}=d_{\psi_2}(\psi_1)-d_{\psi_1}(\psi_2)-[\psi_1,\psi_2]$ contains no terms of the shape $x^n_0x_1$, the function $r_{\{\psi_1,\psi_2\}}$ is by definition 0.
	
	For the calculation of the term $\{\psi_1,\psi_2\}_{x_0}$ and ${}_{x_1}\{\psi_1,\psi_2\}$, we simply use the derivation rule. 
	\begin{align*}
		(d_{\psi_2}(\psi_1))_{x_0}&=(d_{\psi_2}((\psi_1)_{x_0}x_0+(\psi_1)_{x_1}x_1))_{x_0}\\
		&=d_{\psi_2}((\psi_1)_{x_0})+(\psi_1)_{x_1}x_1(\psi_2)_{x_0}\\
		{}_{x_1}(d_{\psi_2}(\psi_1))&={}_{x_1}(d_{\psi_2}(x_1{}_{x_1}(\psi_1)+x_0{}_{x_0}(\psi_1)))\\
		&=\psi_2~{}_{x_1}(\psi_1)-{}_{x_1}(\psi_2)x_1{}_{x_1}(\psi_1)+d_{\psi_2}(x_1(\psi_1))
	\end{align*}
\end{proof}

\begin{Lemma}\label{lem:mid3}
For $\psi_1,\psi_2\in \overline{{\mathfrak{rc}}_0}$, we have the calculation
\begin{align*}
&\mu([\psi_1,\psi_2])\\
=&[\mu(\psi_1),\psi_2]-[\mu(\psi_2),\psi_1]+(\psi_1)_{x_1}x_1{}_{x_1}(\psi_2)+(\psi_1)_{x_0}x_0{}_{x_0}(\psi_2)\\
&-(\psi_2)_{x_1}x_1{}_{x_1}(\psi_1)-(\psi_2)_{x_0}x_0{}_{x_0}(\psi_1)\\
=&[-r_{\psi_1}(x_1)+r_{\psi_1}(x_0)-(\psi_1)_{x_0}-{}_{x_1}(\psi_1),\psi_2]-[-r_{\psi_2}(x_1)+r_{\psi_2}(x_0)-(\psi_2)_{x_0}-{}_{x_1}(\psi_2),\psi_1]\\
&+(\psi_1)_{x_1}x_1{}_{x_1}(\psi_2)+(\psi_1)_{x_0}x_0{}_{x_0}(\psi_2)-(\psi_2)_{x_1}x_1{}_{x_1}(\psi_1)-(\psi_2)_{x_0}x_0{}_{x_0}(\psi_1)
\end{align*}
\end{Lemma}

With all the preparations, for $\psi_1,\psi_2\in \overline{{\mathfrak{rc}}_0}$, we could directly compute
\begin{align*}
&{}_{x_1}(\{\psi_1,\psi_2\})+\mu\{\psi_1,\psi_2\}+(\{\psi_1,\psi_2\})_{x_0}+r_{\{\psi_1,\psi_2\}}(x_1)-r_{\{\psi_1,\psi_2\}}(x_0)\\
=&{}_{x_1}(\{\psi_1,\psi_2\})+\mu(d_{\psi_2}(\psi_1)-d_{\psi_1}(\psi_2)-[\psi_1,\psi_2])+(\{\psi_1,\psi_2\})_{x_0}=0
\end{align*}
The first equality is because of Lemma \ref{lem:mid_1}, the functions $r_{\{\psi_1,\psi_2\}}$ vanishes. The second equality is by the Lemma \ref{lem:mid_1}, \ref{lem:mid_2} and \ref{lem:mid3}. With this equality, we prove that $\overline{{\mathfrak{rc}}_0}$ is closed under Ihara bracket, and it is a Lie algebra.

\end{document}